\documentclass[11pt]{amsart}

\newtheorem{theorem}{Theorem}

\newtheorem{remark}{Remark}

\newtheorem{example}{Example}

\newtheorem{corollary}{Corollary}

\begin{document}
\title[Feedback stabilization]{Feedback stabilization of  linear and bilinear unbounded systems in Banach space}
\author{K. Ammari}
\address{UR Analysis and Control of PDE's, UR 13E64, Department of Mathematics,
Faculty of Sciences of Monastir, University of Monastir, 5019 Monastir, Tunisia}
\email{kais.ammari@fsm.rnu.tn}

\author{ S. El Alaoui}
\address{Department of  mathematics \& informatics, ENS. Univesity of Sidi Mohamed Ben Abdellah. Fes, Morocco}
\email{elalaoui.sfe18@gmail.com}

\author{M. Ouzahra}
\address{Department of  mathematics \& informatics, ENS. Univesity of Sidi Mohamed Ben Abdellah. Fes, Morocco}
\email{mohamed.ouzahra@usmba.ac.ma}

\begin{abstract}
We consider linear control systems of the form \\ $\dot{y}(t)=Ay(t)-\mu B C y(t)$ where $\mu$ is a positive real parameter,
$A$ is the state operator and generates a linear $C_0-$semigroup of contractions $S(t) $ on a Banach space $X$,
$B$ and $C$ are respectively  the operators of control  and  observability, which are defined in appropriate spaces in which they are unbounded in some sense. We aim to show the exponential stability of the above system under sufficient conditions which are expressed in term of admissibility and observability properties. The uniform exponential stabilization using bilinear control is considered as well. Applications to transport and heat equations are also provided.
\end{abstract}


\maketitle

\tableofcontents

\section{Introduction}

Let us consider the following linear control system
\begin{equation}\label{SL}
    \dot{y}(t)=Ay(t) +  B u(t),\; t\ge 0,\; \;\\
    \\
        y(0)=y_0\in X
\end{equation}
augmented with the output $z(t)=Cy(t),\; t\ge 0, $ where $X$ and $U$ are two Banach spaces  representing respectively, the state and observation$/$control space,  $A: D(A)\subset X \rightarrow X $ is the system operator, which  generates a $C_0-$semigroup of contractions $S(t)$ on $X$, the space $X$ is endowed with  norm $\|\cdot\|_X$, and let $X_{-1}$ denote the completion of $X$ w.r.t to the  norm $\|x\|_{-1}:=\|(A-\eta I)^{-1}x\|_X,  \; x\in X$ for some (or equivalently all) $\eta $ in resolvent set  $ \rho(A)$ of $A$, $ B\in {\mathcal L}(U,X_{-1})$ is the control operator and $C\in {\mathcal L}(W,U)$ is the observation operator, where $W$ is a Banach space such that the injections $X_1 \hookrightarrow W \hookrightarrow X$ are continuous ($X_1$ being  the space $D(A)$ equipped  with the graph norm). Then, closing the system (\ref{SL}) with the control $u(t)=- \mu Cy(t), $ ($\mu >0$ is the gain control) one obtains the following  Cauchy problem
\begin{equation}\label{SF}
    \dot{y}(t)=Ay(t) - \mu B Cy(t),\; \; t>0, \;  \;  y(0)=y_0\in X,
\end{equation}
which is well-posed in $X$ whenever  $A-\mu B C$  is a generator of a $C_0$-semigroup on $X$ (cf. \cite{eng}, Section II.6).

\medskip

We further consider the bilinear system
\begin{equation}\label{S}
    \dot{y}(t)=Ay(t) + v(t) \mathcal{B} y(t),\; \\
    \\
    y(0)=y_0\in X\cdot
\end{equation}
The well-posedeness of the  systems like  (\ref{SL}) and (\ref{S}) has been studied in many works using different approaches (see e.g. \cite{adler,ber09,bou05,had05,mar,staf05,weiss94}).\\
In several practical situation, the modeling gives rise to unbounded control systems of form (\ref{SL}) or (\ref{S}), where the closed loop operator is of type   Weiss-Staffans, Miyadera-Voigt or Desch-Schappacher (see e.g.\cite{Kai-Nec,grei,luo,oze,ris,staf05,  voi, weis89}). This fact often occurs  when the control is exercised through the boundary or a point for systems governed by partial differential equations.

\medskip

The problem of feedback  stabilization of some classes of linear and nonlinear systems  has been investigated in case of bounded and unbounded control operators in \cite{FK,ambch1,ambch2,Kai-Nec,amo,amtr,KM1,KM2}. Feedback  stabilization of the  bilinear system  (\ref{S})  has been investigated in the case of a bounded control operator  by numerous authors using various control approaches, such as quadratic control laws, sliding mode control, piecewise constant feedback and optimal control laws (see \cite{bal,  ouz20} and the references therein). Recently, the question of  stabilization of bilinear systems with unbounded control operator has been treated in \cite{ber10,elay12,aya18,ouz17}. In \cite{ber10}, the author considered the case where $A$ is self-adjoint and $B$ is positive self-adjoint and bounded from some subspace $V$ of $H$ to its dual space $V^{\prime},$ then he established the weak and strong stabilizability of the system (\ref{S}) for all $y_0\in D(A)$ using nonlinear control.  Moreover, in \cite{elay12,aya18} it has been supposed that the linear operator $B$ is  relatively  bounded  w.r.t $A$ from  $H$ to an extension $X$ of $H$ with a continuous embedding $H \hookrightarrow
X$. Then, under an exact observability condition, it has been shown that (\ref{S}) is strongly
stabilizable, and a polynomial decay estimate of the stabilized state has been provided in the case of positive self-adjoint control operator.  In \cite{ouz17}, the exponential stabilizability of bilinear systems has been considered for Miyadera's  control operator, and the stabilizing control is a switching one which leads to a closed-loop system like (\ref{SF}) evolving in a reflexive state space. More recently, the case of nonreflexive state space was considered in the context of bounded control operator \cite{ouz20}. In this paper, we  deal with a wide class of linear$/$bilinear systems evolving on a nonreflexive state space with unbounded control operators, including control operators of type Weiss-Staffans, Miyadera-Voigt or Desch-Schappacher. Then we will  give  sufficient conditions for exponential stabilizability of  infinite dimensional systems that can be described by the systems (\ref{SL}) or (\ref{S}).
\medskip

The paper is organized as follows: In the second section,  we provide some tools that will be required for the stabilization problem, then state and show the main result in which we present  sufficient conditions for exponential stabilization of the linear system (\ref{SL}) with a feedback control involving the output, which leads to closed-loop operator of Weiss-Staffans's type. Next, we provide applications to bilinear system (\ref{S}) with control operator of Miyadera-Voigt or Desch-Schappacher type. Applications  to transport and heat equations are presented as well.

\section{The main results }

\subsection{Preliminary on linear semigroups}

Let us recall some notions and properties related to linear $C_0-$semigroups.

\begin{itemize}

\item
The duality pairing between the   space $X$ and
its dual $X^*$ is denoted by  $\left<\cdot,\cdot\right >,$ where $X^{*}$ is  the set of all bounded linear forms on $X$ and the pairing between $y\in X$ and $\phi\in X^{*}$ is denoted by
$\left< y,\phi\right >$.
The duality map $J$ from $X$ to $X^{*}$ is in general a multi-valued operator; i.e. for each $y\in X,$ $J(y)$ is by definition the (nonempty) set of all $\phi\in X^{*}$ such that $\left< y,\phi\right >=\Vert y\Vert_X^{2}=\Vert \phi \Vert^{2}$, where $\Vert \cdot \Vert $ denotes the norm of $X^*$ associated to $\Vert \cdot \Vert_X$.

\item
A one parameter family $S(t), \; t \ge 0,$ of bounded linear operators from a Banach space $X$ into $X$ is a semigroup  on $X$ if (i) $S(0) = I,$ (the identity operator on $X)$ and (ii) $S(t + s) = S(t)S(s) $ for every $t, s\ge 0$. A semigroup $S(t)$ of bounded linear operators on $X$ is a $C_0-$ semigroup if in addition $\displaystyle \lim_{t\to 0^+} S(t)x = x$ for every $x \in X. $ This property guarantees the continuity of the semigroup on $\mathbb{R}^+$. Moreover, one can show (see \cite{paz}, p. 4 ) that for every  $C_0-$ semigroup $S(t)$, there exist constants $\omega \ge 0$ and $M\ge 1$  such that
\begin{equation}\label{s-g}
 \|S(t)\| \le M e ^{\omega t},\; \forall t\ge 0\cdot
  \end{equation}
   If $\omega = 0$ and $M = 1,\, S(t)$ is called  a $C_0-$semigroup of contractions.\\
The linear operator $A$ defined by  $Ax = \displaystyle \lim_{t\to 0^+} \frac{S(t)x - x}{t} $ for $x\in X$ such that $\displaystyle \lim_{t\to 0^+} \frac{S(t)x - x}{t} $ exists in $X,$ is the infinitesimal generator of the $C_0-$semigroup $S(t). $ The   linear space  $D (A) := \{x\in X:\;  \displaystyle \lim_{t\to 0^+} \frac{S(t)x - x}{t} \in X \}$ is the domain of $A$.

The infinitesimal generator of a contraction $C_0-$semigroup  is  dissipative, i.e., for every $y\in D(A)$ and for all  $y^*\in J(y)$ we have  $Re \langle Ay,y^*\rangle \le 0$  (see \cite{paz}, pp. 14-15).

\item For $x \in D(A),$ we have $ Ax= \frac{d^+S(t)x}{dt}|_{t=0} $ and    $ y(t):= S(t)y_0   $ is differentiable and lies in $D(A)$ for all $t>0, $ and is the unique solution of the Cauchy problem: $\dot{y}(t)=Ay(t), t>0, \; y(0)=y_0$. Moreover, for  every $y_0\in X;\; y(t)= S(t)y_0 $ is called mild solution of this Cauchy problem.


\item If $A$ is the infinitesimal generator of a $C_0-$semigroup $S(t),$ then
$D(A)$  (the domain of $A$) is dense in $X$ and $A$ is a closed linear operator.
 Moreover, according to Hille-Yosida's Theorem (see for instance \cite{paz}, p. 20), a linear operator $A$ is the infinitesimal generator of a $C_0-$semigroup $S(t)$ satisfying (\ref{s-g}) if and only if
(i) $A$ is closed and $\mathcal{ D}(A)$ is dense in $X$, and (ii) the resolvent set $\rho(A)$ of $A$ contains the ray $(\omega,+\infty)$ and
$\|R(\lambda, A)^n\| \le \frac{ M}{(\lambda - \omega)^n} $ for $\lambda > \omega, \, n = 1, 2, ...$ In particular,   a closed operator $A$ with densely domain $D(A)$ in $X$ is the infinitesimal  generator of a $C_0-$semigroup of contractions on $X$ if and only if  the resolvent set $\rho(A)$ of $A$ contains $ \mathbb{R}^+ $ and for all $\lambda>0; \; \|\lambda R(\lambda,A)\| \le 1$ (see \cite{paz}, p. 8).

\item The $C_0-$semigroup $S(t)$  may be extended to a $C_0-$semigroup $S_{-1}(t)$ on $X_{-1},$ whose generator is the extension $A_{-1}: D(A_{-1}):=X  \subset X_{-1}  \rightarrow X_{-1}$ of $A: D(A) \subset X  \to X$ to a m$-$dissipative operator from $X$ to $X_{-1}$. In particular, we have  $A_{-1}y=Ay$ for all $y\in D(A)$ (see \cite{eng}, p. 126). Using the integral representation of the resolvent one obtains
$R(\lambda,  A_{-1})y = R(\lambda, A)y,\; \forall y\in X, $ for all $\lambda \in \rho(A_{-1}).$  Recall also that $R(\lambda,  A_{-1})y \in X,\; \forall y\in X_{-1}$ and  $R(\lambda,  A_{-1})y = R(\lambda, A)y,\; \forall y\in X, \; \forall \lambda \in \rho(A_{-1}).$
 Moreover, If $A$ is dissipative, then so is $A_{-1}$: For $z\in X,$ we have
$
\|S_{-1}(t)z\|_{-1}=\|S(t)z\|_{-1}=\|R(\eta:A)S(t)z\|_X\le \|R(\eta:A)z\|_X=\|z\|_{-1}.
$
Then by density of $X$ in $X_{-1}$, we conclude that $S_{-1}(t)$ is a contraction on $X_{-1}$.

\item We have
 $\displaystyle \sup_{\lambda\in \rho(A_{-1})}\|\lambda R(\lambda,A_{-1}) B  \|_{{\mathcal L}(U,X_{-1})} <\infty,$
where $ \rho(A_{-1})$ is the resolvent set   of $A_{-1}$.
Moreover,    by the closed graph theorem we have that $\lambda R(\lambda,A_{-1}) B \in {\mathcal L}(U, X)$. For instance  for $W=U=X_1$ and  $Range (B)\subset X, \; i.e., B\in {\mathcal L}(X_1,X)$ (which is the case of Miyadera's operators), we have for all (real) $\lambda \in   \rho(A_{-1})$ large enough,
$$\|\lambda R(\lambda,A_{-1}) B  \|_{\mathcal{L}(X_1,X)}=\|\lambda R(\lambda,A) B  \|_{\mathcal{L}(X_1,X)} \le
$$
$$
\|B\|_{{\mathcal L}(X_1,X)},\; \forall \lambda >0.$$
In fact this is also true for any admissible operator $B$ in the sense of $(h_2)$ below (see \cite{staf05}, p. 219), that is   there exists $K>0$ such that
   \begin{equation}\label{k}
   \|\lambda R(\lambda,A_{-1}) B  \|_{\mathcal{L}(U,X)}\le  K, \; \; \mbox{for all } \; \lambda \; \mbox{large enough}.
   \end{equation}

In general this property does not imply the admissibility of $B$.
 (see \cite{bou05,bou10,bou11,bou14,haak,jac03,mar,merd,weis-conj} for some discussions and partial results about this implication).

\item  Let  $\mathfrak{X} \oplus \mathfrak{X}_{-1}$ be a direct (algebraic) decomposition  in $X_{-1}$ between two subspaces $ \mathfrak{X} $ and $ \mathfrak{X}_{-1}$ such that  $ \mathfrak{X} \subset X$ and $X \cap \mathfrak{X}_{-1}=\{0\}$, and let $P_{\mathfrak{X}}$ denote the projection on $\mathfrak{X}$ according to the above decomposition.
 Moreover, if $K$ is a linear  operator such that  $Range(K)\subset\mathfrak{X} \oplus \mathfrak{X}_{-1},$ then  we set  $_{_X}\!K := P_{\mathfrak{X}} K.$ \\
Note that $_{_X}\!K$ depends on the choice of $\mathfrak{X} $ and $ \mathfrak{X}_{-1}, $ so in the sequel, we suppose that such a choice is made. In this case, we have $Range(K)=Y\oplus Z$ with $Y:=\mathfrak{X}\cap Range (K)\subset X$ and $Z:=\mathfrak{X}_{-1}\cap Range (K)$ so $Z\cap X=(0)$. Moreover, we can also write
$K=K_1+K_2$ with $Range(K_1)\subset X$ and $Range(K_2)\cap X= \left\{0\right\},$ and we have $K_1=_{_X}\!\!K.$
 \end{itemize}
\subsection{The  stabilization results}

In this part we consider the stability of the system (\ref{SF}).  The first task is to guarantee the existence and uniqueness
 of the solution. Note that  if the operator $\mu BC\in {\mathcal L} (W,X_{-1})$ is a Weiss-Staffans perturbation for $A$,  then (see e.g. \cite{adler,weis89}) the  closed loop system (\ref{SF}) is well-posed. More precisely, for small gain control $\mu>0$, the operator $(A-\mu BC)|_X$ (i.e. the part of $A_{-1}-\mu BC$ on $X$) with domain $D_\mu:=\{y\in W:\; (A_{-1}-\mu BC)y\in X\}$ generates a $C_0-$semigroup $T(t)$ on $X$ satisfying the following variation of constants formula (V.C.F)
 \begin{equation}\label{TBC}
T(t)y_0=S(t)y_0-\mu \int_0^t S_{-1}(t-s) BC T(s) y_0,\; \forall y_0\in D_\mu.
\end{equation}
Note that in general, we have $D(A)\cap D((BC)|_X)\subset D_\mu.$ Moreover, if $W\subset X_1\cup D((BC)|_X),$ then we have  $(A_{-1}-\mu BC)|_X=D(A)\cap D((BC)|_X).$


\medskip

This motivates the consideration  of the  assumptions  $(h_1)-(h_4)$ below.

\medskip

$(h_1)$ the well-posedness assumption: there exists $ \alpha_1>0$ such that  for every $\mu \in (0,\alpha_1), $  the operator $(A_{-1}-\mu BC)|_X$ with domain $D:=D((A_{-1}\cap BC)|_X)$ generates a $C_0-$semigroup $T(t)$ on $X$.

%

\medskip

Now let us consider the following assumptions for some  $T, M >0$.

\medskip

  $(h_2)$ The admissibility assumption of $ B\in {\mathcal L} (U,X_{-1})$:    $$\displaystyle \int_0^T S_{-1} (T-s) Bu(s)  ds \in X, \, \forall \, u\in  L^1(0,T;U), $$
which  implies that
$$\left\|\displaystyle \int_0^T S_{-1} (T-s) Bu(s)  ds\right\|_X \le M\|u\|_1, $$
for all $u\in L^1(0,T;U)$ (or equivalently for all $ u \in  W^{1,1}(0,t_1;U)$). This also implies that the operator  defined by
$$\mathcal{B}_T : u(\cdot)\in L^1(0,+\infty;U) \mapsto \int_0^T S_{-1} (T-s) Bu(s) ds$$
 is bounded (see  \cite{adler}).

\medskip

Note that  if $X$ is reflexive, then the admissibility assumption $(h_2)$ is equivalent to the boundedness  of $B$ (see \cite{weis89}).
\medskip

  $(h_3)$ The admissibility assumption of $ C\in {\mathcal L} (W,U)$:
  $$\int_0^T \|CS(t)y\|_U dt \le M\|y\|_X,\; \forall y\in D(A). $$

\medskip

$(h_4)$ Joint-admissibility  of $B$ and $ C$: $$\int_0^T \left\|C \int_0^r S_{-1}(r-s)Bu(s) ds  \right\|_U dr \le M\|u\|_1,\; \forall u\in L^1(0,T;U)$$
   with
  $\|u\|_1=\int_0^T \|u(\tau)\|_X d\tau.$

\medskip

In the sequel, if there is no confusion, we use  $\langle z,J(y)\rangle$ for any $ y^*\in J(y)$ instead  of $\langle z,y^*\rangle$. Also, for any functions $t\mapsto \zeta(t)$, we will write $\phi(\cdot)\in J(\zeta(\cdot))$ if $\phi(t)\in J(\zeta(t)),\, \forall
t\ge0.$

\medskip

Now, for the stabilization results we further consider the  following observation assumption.

$(h_5)$ The observability condition: for some $ \delta >0$ we have
\begin{equation}\label{obs}
\int_0^T \mathcal{R}e \langle  _{_X}\!(BC)S(t)y,  J(S(t)y) \rangle dt\ge \delta \| S(T)y\|_X^2,\; \forall y\in D(A).
\end{equation}

The  estimate (\ref{obs})  may be seen as a null-exact controllability inequality in the sense of linear systems.

\medskip

$(h_6)$  The  function $F_J: y\mapsto \{\langle  _{_X}\!(BC) y,y^*\rangle; \; y^*\in J(y)\}$ is such that
   for all $y, z \in X$, there exists $(y^*,z^*)\in J(y)\times J(z)$ such that
 $$
|\langle _{_X}\! (BC) y,y^*\rangle-\langle _{_X}\!(BC) z,z^*\rangle| \le  k_1 \left( \|y\|_{D(C)}+ \|z\|_{D(C)} \right)) \; \|y-z\|_X  +
	$$
	$$
   k_2 ( \|y\|_X + \|z\|_X) \; \|C(y-z)\|_U, \;\;\; \forall y, z\in D(A)\subset W
    $$
  for some  constants $k_i\ge 0,\; i=1,2$, where $\|y\|_{D(C)}:= \|y\|_X+\|Cy\|_U$.

\medskip

This assumption is motivated by the fact that here,  the state space is a general Banach space, i.e. without any smoothness property that evolves  the duality mapping. In particular, if the state space $X$ is smooth, so that $J$ is Lipschitz-continuous and  $X$ is reflexive (see e.g. \cite{ouz17}), then $(h_6)$ is verified under the admissibility  of $B$, as in that case the operator $BC$ will be bounded from $W$ to $X$ (see \cite{weis89}).

\medskip

Let us now state our main result.

\begin{theorem} \label{thm1}
Let  assumptions $(h_1)-(h_6)$ hold.
Then there exists $\alpha>0$ such that for any $\mu \in(0,\alpha)$, the closed-loop system (\ref{SF}) is exponentially stable.
\end{theorem}

\begin{proof}
According to assumption $(h_1)$,  the operator $A_{BC}:=(A_{-1}-\mu BC)|_X$  with domain $D(A_{BC})=D(A)\cap D((BC)|_X)=D$ generates a $C_0-$semigroup $T(t)$ on $X$ for $\mu>0$ small enough (says $\mu \in (0,\alpha_1)$). Moreover, $y(t):=T(t)y_0$ is the unique mild solution of (\ref{SF}) and satisfies the following V.C.F
 \begin{equation}\label{vcf}
y(t)=S(t)y_0-\mu \int_0^t S_{-1}(t-s) BC y(s) ds,\, \forall \, y_0\in D.
\end{equation}
Let $y_0\in D$ be fixed. Then for all $t\ge 0, $ we have   $(A_{-1} -\mu BC)|_Xy (t)\in X.$ Moreover, $y(t)$ has a  weak derivative $A_{BC}y(t)=T(t)A_{BC}y_0$, which is weakly continuous and hence bounded $$|\frac{d}{dt} \langle y(t),f\rangle |\le L \| A_{BC}y_0\|_X \|f\|,\; \forall f\in X^* \;\; (L>0)$$ in any bounded time-interval. Thus $y(t)$ (and so is $\|y(t)\|_X$) is Lipschitz continuous (recall that $\|y\|_X^2= \displaystyle \sup_{f\in X^*, \|f\|\le 1 } |\langle y,f\rangle|$). It follows that $\Vert y(t)\Vert_X$ is differentiable almost everywhere and  (see \cite{kato}) for a.e $t>0$ we have
 \begin{equation}\label{dy}
\frac{d}{dt}\Vert y(t)\Vert_X^2= 2 \mathcal{R}e \langle A_{BC}y(t),J(y(t))\rangle.
\end{equation}
Here $J$ is the duality  mapping of $X$ (recall that $ A_{BC}y(t)\in X),$ which by integrating and using the dissipativeness of $A$ implies
 \begin{equation}\label{iner}
2\mu \int_s^t \mathcal{R}e  \langle BCy(\tau), J( y(\tau)) \rangle d\tau \le \|y(s)\|_X^2-\|y(t)\|_X^2,\; t\ge s\ge 0.
\end{equation}
According to  $(h_6),$ we have
$$
  \mathcal{R}e \langle _{_X}\!(BC) S(t)y_0,  J(S(t)y_0) \rangle \le
	$$
	$$
	K \left( \|S(t)y_0\|_{D(C)} + \|y(t)\|_{D(C)}  \right) \|S(t)y_0-y(t)\|_X+
	$$
	$$
  K \left(\|S(t)y_0\|_X +\|y(t)\|_X \right) \|C ( S(t)y_0-y(t))\|_U +   $$
	$$
  \mathcal{R}e \langle  BC y(t), J(y(t))\rangle
$$
where  $K=max(k_1,k_2)$, where we have used that $_{_X}\!(BC) y(t)=BC y(t),$ as $y(t)\in X, \; \forall t\ge 0.$

\medskip

From the admissibility assumption $(h_3)$, we have
$$
\int_0^T \|C S(t)y_0\|_{U} dt \le M\|y_0\|_X
$$
and for all $t\in [0,T]$, we have
$$ \|S(t)y_0-y(t)\|_X  =\mu \left\|\int_0^t S_{-1}(t-s) BC y(s) ds \right\|_X $$
$$
\le \mu \|\mathcal{B}_T (C y(\cdot)) \|_1,
$$
where  $\mathcal{B}_T$ is the  bounded operator defined by
$$\mathcal{B}_T : u(\cdot)\in L^1(0,+\infty;U) \mapsto \int_0^T S_{-1} (T-s) Bu(s) ds.$$
Hence
$$
\|S(t)y_0-y(t)\|_X \le \mu M \|Cy(\cdot)\|_1
$$
where
$\|Cy(\cdot)\|_1:=\int_0^T \|Cy(\tau)\|_U d\tau.$\\
It follows  from this and $(h_4)$ that
$$\int_0^T \left\|C ( S(t)y_0-y(t))\right\|_{U} dt =\mu \int_0^T \left\| C\int_0^t S_{-1}(t-s) BC y(s) ds \right\|_U dt$$
$$
\le \mu M \|Cy(\cdot)\|_1.
$$
Let us estimate $\|Cy(\cdot)\|_1.$ For every $t\ge 0$, we have $y(t)\in D\subset W $ and $S(t)y_0\in D(A)\subset W.$ Then, from the V.C.F, we derive
$$
\begin{array}{lll}
  \int_0^T \|Cy(\tau)\|_U d\tau &\le \int_0^T \|CS(\tau)y_0\|_U d \tau+ \mu \int_0^T \|C\int_0^t S_{-1}(t-s) BC y(s)ds  \|_U dt
&  \\
  &\le M\|y_0\|_X+ M \mu \|C y(.)\|_1&\\
  &= M\|y_0\|_X+M\mu \int_0^T \|Cy(\tau)\|_U d\tau.
 &
\end{array}
$$
Hence  for $0<\mu < \alpha_2:=\inf(\alpha_1, \frac{1}{M})$, we have
\begin{equation}\label{Cy}
\int_0^T \|Cy(\tau)\|_U d\tau \le \frac{M}{1-M \mu} \|y_0\|_X.
\end{equation}
Then $$\|y(t)\|_X \le \left(1+\frac{\mu M^2}{1-M \mu} \right) \|y_0\|_X.$$
We conclude that
$$\int_0^T \mathcal{R}e \langle   _{_X}\!(BC) S(t)y_0,J(S(t)y_0)\rangle dt \le c \mu \|y_0\|^2_X+
\int_0^T \mathcal{R}e \langle BC y(t), J(y(t)) \rangle dt
$$
for some constant $c>0$ which is independent of $y_0$.\\
This together with (\ref{obs}) gives
$$
\delta \|S(T)y_0\|_X^2 - c\mu \|y_0\|^2_X\le \int_0^T \mathcal{R}e \langle  BC y(t), J( y(t) ) \rangle dt.
$$
Then taking $y(t)$ instead of $y_0$ in the last estimate, it comes
\begin{equation}\label{*S(T)}
\delta \|S(T)y(t)\|_X^2 - c\mu \|y(t)\|^2_X \le \int_t^{t+T} \mathcal{R}e \langle BC y(s), J( y(s)) \rangle ds.
\end{equation}
From  the  variation of constants formula (\ref{vcf})  we deduce that  for all $t\ge T,$ we have
  $$ \begin{array}{lll}
       \|y(t)\|_X &\le   \|S(T)y_0\|_X+ \mu \|\int_0^t S_{-1}(t-s) BC y(s) ds\|_X& \\
      & \le \|S(T)y_0\|_X+ \mu M \|Cy(\cdot)\|_1.&
      \end{array}
$$
Then taking $y(kT)$ instead of $y_0$, it comes via (\ref{Cy})
$$ \begin{array}{lll}
     \|y(t)\|_X
 &\le & \|S(T)y(kT)\|_X+ \mu M \int_{kT}^{(k+1)T} \|Cy(s)\|_U ds
 \\
     &\le & \|S(T)y(kT)\|_X+  \frac{\mu M^2}{1-M \mu}   \|y(kT)\|_X,\; \forall t\in [kT,(k+1)T].
   \end{array}
  $$
Thus for all $k\ge0,$ we have
\begin{equation}\label{zS(T)}
  \|y((k+1)T)\|_X^2
\le 2 \|S(T)y(kT)\|_X^2 +  2 \big ( \frac{\mu M^2}{1-M \mu} \big )^2  \|y(kT)\|^2_X.
  \end{equation}
This together with (\ref{iner}) and (\ref{*S(T)}) implies
$$
\mu \delta\bigg ( \| y((k+1)T)\|_X^2 - 2 \big ( \frac{\mu M^2}{1-M \mu}  \big )^2 \|y(kT)\|^2_X \bigg ) - 2c\mu^2 \|y(kT)\|^2_X\le
$$
$$
\| y(kT)\|_X^2 -\| y((k+1)T)\|_X^2.
$$
Hence
$$
(1+\mu \delta)   \| y((k+1)T)\|_X^2 \le  \bigg (2\delta \mu ( \frac{ \mu M^2}{1-M \mu} )^2 +2c\mu^2  + 1 \bigg ) \| y(kT)\|_X^2 , \; k\ge 0,
$$
from which we derive
\begin{equation}\label{y(kT)}
\| y(kT)\|_X^2 \le q^k \| y_0\|_X^2,\; \forall k\ge 0
\end{equation}
with $q:=\frac{1+ 2\mu^2 \left (  \delta \mu \left(\frac{ M^2}{1-M \mu}\right)^2 +c  \right) }{1+2\mu \delta}, $ which lies in $(0,1)$ for $\mu\to0^+.$  Moreover, using the following  well known property of linear $C_0-$semigroups:
$$
\| y(t)\|_X\le N e^{wt }\| y_0\|_X,\; t\ge 0,\; (\mbox{for some constants}\; N, w>0),
$$
 we deduce, by taking $k=E(t/T)$, that $\|y(t)\|_X\le N \,e^{wT} \|y(kT)\|_X$ and hence by (\ref{y(kT)})
\begin{equation}\label{y(t)}
\| y(t)\|_X \le M' e^{-\sigma t} \| y_0\|_X,\; \forall t\ge 0,
\end{equation}
where $M', \sigma$ are independent of $y_0$.\\
This estimate extends to all initial data in $X$ by density of
$D$ in $X$.
\end{proof}

In the previous theorem, we have considered the case where the domain of the generator  $A_{BC}$ is $D(A)\cap D((BC)|_X).$ In the next result,  we will state  an other stabilization result which only requires that the domain of the generator  $A_{BC}$ is independent of the gain control $\mu,$ provided some further conditions are fulfilled. Let us consider the following assumption.

  \medskip
  $(h_7)$ Compatibility condition: Range $( B_\lambda) \subset W,$ for some$/$all $\lambda\in \rho(A) $ holds with $B_\lambda:=\lambda R(\lambda,A_{-1}) B$.
	
	\medskip


 Notice that under the compatibility assumption,  it comes from the
closed graph theorem and the resolvent property that  $Range ( B_\lambda C) \subset \mathcal{W}. $
 Let us define the following scalar valued function $$f_{_{BC}}(y)=\limsup_{\lambda \to +\infty} \mathcal{R}e  \langle B_\lambda Cy,J(y)\rangle.$$
 Note that by (\ref{k}), we have $ f_{_{BC}}(y) \in \mathbb{R},\; \forall y\in X.$ Moreover,  if $\overline{BC}$ is the operator defined by
 $$\overline{BC}y :=\lim_{\lambda \to +\infty}   B_\lambda Cy, \;  \forall y\in D(\overline{BC}):=\{y\in W:\; \lim_{\lambda \to +\infty}   B_\lambda Cy \;\mbox{exists in}\; X \},$$ then we have
\begin{equation}\label{fBC}
f_{_{BC}}(y)=\mathcal{R}e  \langle\overline{BC}y,J(y)\rangle,\; \forall y\in  D(\overline{BC}).
\end{equation}
If in addition  $Range(B)\subset X,$ then the relation (\ref{fBC}) holds in $X$.

 Let us consider  the following   assumptions:

\medskip

$(h_5)'$ The observability condition: for some $ \delta >0$ we have
\begin{equation}\label{obsbis}
\int_0^T f_{_{BC}}(S(t)y) dt\ge \delta \| S(T)y\|_X^2,\; \forall y\in D(A),
\end{equation}

\medskip
$(h_6)'$   for all $y, z \in X$, there exists    $(y^*,z^*)\in J(y)\times J(z)$ such that
 $$
 \begin{array}{lll}
   \left| f_{_{BC}}(y)-f_{_{BC}}(z)\right| &\le  k_1 \left( \|y\|_{D(C)}+ \|z\|_{D(C)} \right) \; \|y-z\|_X  +& \\
   \\
   &k_2 ( \|y\|_X + \|z\|_X) \; \|C(y-z)\|_U, \;\;\; \forall y, z\in D(A)\subset W&
 \end{array}
    $$
  for some  constants $k_i\ge 0,\; i=1,2$.

\medskip
From the proof of Theorem \ref{thm1} we deduce the following result.

\begin{corollary} \label{cor0}

Assume that for some $ \mu_1>0$, the domain of $A_{BC}:=(A_{-1}-\mu BC)|_X$ is independent of $\mu\in  (0,\mu_1),$ and let  assumptions  $(h_2)-(h_4)$, $(h_5)'-(h_6)'$ and $(h_7)$ hold.

Then there exists $\alpha>0$ such that for any $\mu \in(0,\alpha)$, the closed-loop system (\ref{SF}) is exponentially stable.
\end{corollary}

\begin{proof}

Under the  assumptions  $(h_2)-(h_4)$ and the compatibility condition $(h_7)$, the operator $\mu BC\in {\mathcal L} (W,X_{-1})$ is a Weiss-Staffans perturbation for $A$. Accordingly (see \cite{adler}), the closed loop operator $(A_{-1}-\mu BC)|_X$ with domain $D(A_{BC})=\{y\in W:\; (A_{-1}-\mu BC)y\in X\}$ generates a $C_0-$semigroup $T(t)y_0$ on $X$ and $y(t):=T(t)y_0$ is the unique mild solution of (\ref{SL}) and satisfies the formula (\ref{TBC}) for  all $y_0\in D(A_{BC}).$

\medskip

Let $y_0\in D(A_{BC})$ be fixed. Then  we have   $(A_{-1} -\mu BC)|_Xy (t)\in X$ for all $t\ge 0,$ and $\Vert y(t)\Vert_X$ is differentiable almost everywhere and (\ref{dy}) holds for $y_0\in D(A_{BC})$.
Moreover, for every $y\in D(A_{BC}),$ we have
$$
\mathcal{R}e \langle \lambda R(\lambda,A_{-1}) A_{BC}y,J(y)\rangle= \mathcal{R}e \langle \lambda R(\lambda,A_{-1}) A_{-1} y,J(y)\rangle -\mu \mathcal{R}e \langle B_\lambda Cy,J(y)\rangle
$$
Because $A$ is dissipative, we have for all $y\in X$
$$
\mathcal{R}e \langle  R(\lambda,A_{-1}) A_{-1} y,J(y)=\mathcal{R}e \langle -y + \lambda R(\lambda,A)  y,J(y)\rangle
$$
$$
\le 0.
$$
It follows that
$$
\mathcal{R}e \langle \lambda R(\lambda,A_{-1}) A_{BC}y,J(y)\rangle \le -\mu \mathcal{R}e \langle B_\lambda Cy,J(y)\rangle.
$$
Since $y\in D(A_{BC}),$ it comes that $$
\lambda R(\lambda,A_{-1}) A_{BC}y =\lambda R(\lambda,A) A_{BC}y  \to A_{BC}y\: in \;  X, \: as \; \lambda\to +\infty.$$
Hence
$$
\mathcal{R}e \langle  A_{BC}y,J(y)\rangle\le  -\mu  \limsup_{\lambda \to +\infty} \mathcal{R}e  \langle B_\lambda Cy,J(y)\rangle.
$$
$$
=  -\mu   f_{_{BC}}(y).
$$
The remainder of the proof is exactly the same as in the proof of Theorem \ref{thm1}, which leads to the estimate (\ref{y(t)}) with constants  $M', \sigma$ which are  independent of $y_0$. Then we conclude by density of $D(A_{BC})$ in $X$.

\end{proof}

\begin{remark}\label{rem}

 From the proof of Theorem \ref{thm1} (resp. Corollary  \ref{cor0}), we can see  that  the results  remain true if we assume that $(h_5)$ (resp. $(h_5)'$ ) holds for some element of the duality set $J(S(t)y)$ provided that $(h_6)$ (resp. $(h_6)'$) holds for every $(y^*,z^*)\in J(y)\times J(z)$.
\end{remark}

\section{Applications}
In this section we will apply the result of the previous section  to the bilinear system (\ref{S}). More precisely,  we will investigate the exponential stability of (\ref{S}) under the bang-bang  feedback control  $v(t)=-\mu {\bf 1}_{\{t\ge 0;\; \mathcal{B} y(t)\ne0\}}.$ As special cases,
we will consider  Miyadera-Voigt's and Desch-Schappacher's control operators.

\subsection{Stabilization of unbounded bilinear systems}

First let us note that if $\mathcal{B}$ is decomposable according to $\mathcal{B}=BC$ with $B$ and $C$ satisfying the conditions of Theorem \ref{thm1}, then we can see that  (\ref{S}) is exponentially stabilizable by the control $v(t)=-\mu {\bf 1}_{\{t\ge 0;\; \mathcal{B} y(t)\ne0\}}.$

\medskip

 In the  following corollary, we provide a result that extends  the one  of \cite{ouz17} to the case of non reflexive state space.

\begin{corollary}\label{cor3}
Let $ \mathcal{B} \in {\mathcal L} (X_1,X)$ be such that

$
(m_1) $ there exists $M >0$ such that
$$ \int_0^T \|\mathcal{B} S(t)y\|_X dt \le M\|y\|_X,\; \forall y\in D(A),$$

$
(m_2)\; $ there exists $\delta >0$ such that
$$ \int_0^T Re \langle \mathcal{B} S(t)y,J(S(t)y)\rangle dt\ge \delta \|y\|_X^2,\; \forall y\in D(A),$$

$ (m_3)$      the  function $F_J: y\mapsto \{\langle  \mathcal{B} y,y^*\rangle; \; y^*\in J(y)\}$ is  such that
   for all $y, z \in X$, there exists    $(y^*,z^*)\in J(y)\times J(z)$ such that
 $$\begin{array}{lll}
     |\langle  \mathcal{B} y,y^*\rangle-\langle  \mathcal{B} z,z^*\rangle| &\le    k_1 \big ( \|y\|_{D(\mathcal{B})}+ \|z\|_{D(\mathcal{B})} \big) \; \|y-z\|_X  +& \\
           & k_2 ( \|y\|_X + \|z\|_X) \; \|y-z\|_{D(\mathcal{B})}, \;\;\; \forall y, z\in D(A),&
   \end{array}
$$
  for some  constants $k_i\ge 0,\; i=1,2$.

Then  there exists $\alpha>0$ such that for any $\mu \in(0,\alpha), $ the control $v(t)=-\mu {\bf 1}_{\{t\ge 0;\; \mathcal{B} y(t)\ne0\}}$ guarantees the  exponential stabilization  of (\ref{S}).
\end{corollary}

\begin{proof}
Let us first observe that under the assumption of the corollary, the operator $\mu \mathcal{B} $   may be seen as a   Miyadera's perturbation of $A$, for $\mu>0$ small enough, and we have   $D((A_{-1}-\mu \mathcal{B})|_X) = D(A)$.\\
Moreover, we have $_{_X}\!\mathcal{B}=\mathcal{B}$ relatively to the choice  $\mathfrak{X}=X$ and $\mathfrak{X}_{-1}=\{0\}. $

\medskip

Let us take $B=i: X \hookrightarrow X_{-1}$ (the embedding $ X \hookrightarrow X_{-1}$) with $U=X$ and $W=X_1$, so that $\mathcal{B}=C\in \mathcal{L}(X_1,X)$.
Then, since   $Range(\mathcal{B})=Range(C)\subset X,$ the compatibility condition $(h_7)$ follows from the fact that $R(\lambda,A_{-1}) C=R(\lambda,A) C,$  while the others conditions of Theorem \ref{thm1} are clearly satisfied. Moreover, the well-posedeness follows from \cite{miy,voi} (see also \cite{eng}, p. 199) and (\cite{adler}, Theorem 18 and its remark).\\
Finally, by observing that $v(y)\mathcal{B} y=-\mu \mathcal{B} y$ with $v(y)=-\mu {\bf 1}_{(y\not\in \ker \mathcal{B})}$, we can see that the bilinear system (\ref{S}), closed with the feedback control $v(t)=-\mu {\bf 1}_{\{t\ge 0;\; \mathcal{B} y(t)\ne0\}},
$ leads to the system in closed-loop (\ref{SF}). Hence according to Theorem \ref{thm1}, we have the exponential stability for small gain control $\mu>0$.

\end{proof}

We  have the following result regarding the case of Desch-Schapacher's   control operator.

\begin{corollary} \label{cor2} Let $\mathcal{B}\in \mathcal{L} (X,X_{-1})$ be such that for some $\mu_1, T, \beta >0$ we have

$(ds)_1\; $ the domain of $A_{BC}:=(A_{-1}-\mu BC)|_X$ is independent of $\mu\in  (0,\mu_1),$

$(ds)_2\; $   for all $u\in  L^1(0,T;X),$ we have $\displaystyle \int_0^T S_{-1} (T-s) \mathcal{B} u(s) ds\in X, $

$(ds)_3 \; $ there exists  $\delta >0$ such that $$\int_0^T \limsup_{\lambda \to +\infty} Re \langle   \mathcal{B}_\lambda  S(t)y,  J( S(t)y)\rangle dt\ge \delta \|S(T)y\|_X^2,\; \forall y\in X,$$

$ (ds)_4 \; $         for all $y, z \in X$, there exists    $(y^*,z^*)\in J(y)\times J(z)$ such that
 $$ |\limsup_{\lambda \to +\infty} \langle  \mathcal{B}_\lambda y,y^*\rangle- \limsup_{\lambda \to +\infty} \langle  \mathcal{B}_\lambda z,z^*\rangle| \le  k \big ( \|y\|_X+ \|z\|_X \big) \; \|y-z\|_X, \; \forall y, z\in X,$$
  for some  constant $k\ge 0$.

Then  there is $\alpha>0$ such that for any $\mu \in(0,\alpha), $ the control $v(t)=-\mu {\bf 1}_{\{t\ge 0;\; \mathcal{B} y(t)\ne0\}}
$ guarantees the  exponential stabilization  of (\ref{S}).

\end{corollary}

\begin{proof}
This follows from Corollary \ref{cor0} by taking   $U=W:=X$, $C:=I_{X}$ and $B:=\mathcal{B}\in \mathcal{L}(X,X_{-1})$.

\end{proof}

\subsection{Examples}

\begin{example}

Let    us consider the following system
\begin{equation}\label{E1}
\left\{
\begin{array}{ll}
  y_t(\cdot,t) =   y_{x}(\cdot,t)-\mu  \mathbf{B}y(\cdot,t),& \mbox{in} \; (0,1)\times (0,+\infty), \\
  y(1,t)=0, & \mbox{in} \, (0,+\infty),\\
  y(\cdot,0)=  y_0\in L^1(0,1), & \mbox{in} \;(0,1).
\end{array}
\right.
\end{equation}
Here,  $X=L^{1} (0,1)  $ and  the duality map is given  for all $y\in X$ by
$$J(y)=\{\xi\in L^\infty(\Omega) : \; \xi(x)\in  \,\, \mbox{sign}(y(x)) \cdot \|y\|\},$$
 where  the sign function is defined by\\
$  sign (s)=
 \left\{
   \begin{array}{cc}
     1, & s>0, \\
    I, & s=0, \\
     -1, & s<0.
   \end{array}
 \right. $  with $I=[-1,1]$.\\
The state operator is defined  by $ Ay= y'$ with domain $D(A)=W_0^{1,1}(0,1)=\{y\in W^{1,1}(0,1):\; y(1)=0 \}$ and generates the semigroup $ S(t) $  defined  for all $y\in L^1(\Omega)$ by
$$
\big ( S(t) y \big) (\xi)=\left\{
  \begin{array}{ll}
    y(\xi+t), & \mbox{if} \; \xi+t\le 1 \\
\\
    0, & \mbox{else.}
  \end{array}
\right.
$$
 Let $\phi = \alpha \delta_1 \in  \big ( W^{1,1}(0,1) \big )'$  where $\delta_1$ is the Dirac point evaluation in $1$ and $\alpha\in \mathbb{R},$ and let us define the control operator for $y\in  W^{1,1}(0,1)$ by  $\mathbf{B}y=y+\phi (y) A_{-1}a,$ where $a(x)= 1, \; a.e. \; x\in (0,1).$
\\
Let us consider the unbounded part   of $\mathcal{B}$, which is defined by  $ y\mapsto  \phi (y) A_{-1}a, $ and  which may be written in the form $BC:  W^{1,1}(0,1) \rightarrow X_{-1}$, where  $B\in\mathcal{L}(\mathbb{C},X_{{-1}})$  is defined by $Bq=qA_{-1}a,\; q\in U:=\mathbb{C}$ and $C\in\mathcal{L}(W,\mathbb{C})$ is defined by  $ Cy=\phi (y)=\alpha y(1),\: \forall y\in W:=W^{1,1}(0,1).$

\medskip

The admissibility properties of $B$ and $C$ as well as the compatibility condition can be  checked for $T=1$ (see   \cite{adler}), which implies   the well-posedness of the above system (for $\mu>0$ small enough). Let $\mu\in (0,\mu_1)\subset (0,1)$ with $0<\mu_1<\frac{1}{\alpha}$ for which the well-posedness is guaranteed. Then
$$
\begin{array}{lll}
  y\in D((BC)|_{X}) &\Leftrightarrow& \phi(y)A_{-1}a \in X \\
&\Leftrightarrow & \phi(y)=0 \;\; (\mbox{because} \;\; a\not\in D(A))
  \\
& \Leftrightarrow& \mathcal{B} y=y.
\end{array}
$$
Here, one can take  $\mathfrak{X}=W^{1,1}(0,1)$ and $\mathfrak{X}_{-1}=span(A_{-1}a).$ Then we have $_{_X}\!\mathcal{B}y=y,\; \forall y\in W^{1,1}(0,1)$ and so  $\langle _{_X}\!\mathcal{B}y,J(y) \rangle = \|y\|^2,\; \forall y\in W.$  Thus the assumptions  $(h_5) \; \&\; (h_6)$ clearly  hold.\\
Moreover, we have
$$
\begin{array}{lll}
  y\in D(( A_{-1} -\mu \mathcal{B})|_X) &\Rightarrow & A_{-1} ( y-\mu \phi(y) a)  \in X  \\
  &\Rightarrow& y-\mu  \phi(y) a \in D(A)\\
  &  \Rightarrow & y(1)=\mu\phi(y)
  \\
& \Rightarrow & \mathcal{B}y=y\in X  \; (\mbox{recall that}\; 0<\alpha\mu<1)\\
&\Rightarrow&  y\in D(\mathcal{B}|_X)\cap D(A).
\end{array}
$$

\medskip

We conclude by Theorem \ref{thm1} that the system (\ref{E1}) is exponentially stable  for $\mu>0$   small enough.

\end{example}

\begin{example}
Let   $\Omega =\left( 0,+\infty\right)$ and let us consider the following system
\begin{equation}\label{E2}
\left\{
\begin{array}{ll}
  y_t(\cdot,t) =   -y_{x}(\cdot,t)+v(t)  (1+k(x)) y(\cdot,t),& \mbox{in} \;  (0,+\infty)^2 \\
  y(0,t)=0, & \mbox{in} \, (0,+\infty)\\
  y(\cdot,0)=  y_0\in L^1(0,+\infty), & \mbox{in} \;(0,+\infty)
\end{array}
\right.
\end{equation}
Here,  $X=L^{1}\left( \Omega \right) $ is the state space, the parameter $u(t)$ is the bilinear  control and the corresponding solution $z(t):=y(\cdot,t)\in X$ is the state. The function $k$ is such that $k\in L^1(0,+\infty)$  and $\|k \|_X <1.$\\
The unbounded operator $A=-\frac{\partial }{\partial x}$  with domain $$D\left(
A\right) =\left\{ y\in W^{1,1}\left( \Omega \right); \; y\left( 0\right)
=0\right\} $$  generates a group of isometries $S(t)$ on  $X$, which is defined  for all $y\in L^1(\Omega)$ by
$$
S(t) y(\xi)=\left\{
  \begin{array}{ll}
    y(\xi-t), & \mbox{if} \; \xi-t\ge 0 \\
\\
    0, & \mbox{else.}
  \end{array}
\right.
$$
Let us define the operators  $\mathcal{B}=y+k(x)y$ and  $By=k(x)y$.

\medskip

Note that if $k\not\in  L^\infty(0,+\infty),$ then  $\mathcal{B}$ is not a bounded operator on $X$.
\medskip

Let us show that $B$ is $A-$bounded. It comes from
$$
|y(x)|= \left|\int_0^x y'(s) ds \right|\le  \int_0^\infty |y'(s)| ds, \forall y\in D(A)
$$
that
$D(A)\subset D(B))$ and that for all $y\in D(A),$ we have
$$
\|ky\|_X=\int_0^\infty |k(x)y(x)| dx\le \|y\|_{D(A)}\|k\|_X.
$$
Hence $B\in \mathcal{L}(X_1,X).$

$\bullet$ Admissibility of $B$. Let $T>0$,
$$\begin{array}{lll}
    \int_0^T \|BS(t)y\|_X dt &=& \int_0^T \|k(x)S(t)y) \|_X dt
   \\
   &=& \int_0^T \int_0^\infty |k(x)y(x-t)| {\bf 1}_{(0\le t\le x)} dx dt
  \\
& = &\int_0^T \int_0^\infty |k(x+t)y(x)| dx dt\\
&=&\int_0^\infty |y(x)| \left( \int_0^T  |k(x+t)| dt \right) dx\\
&\le& \|k\|_X \|y\|_X.
  \end{array}
  $$
This implies the admissibility etimate $(m_1)$.

$\bullet $ Assumption $(h_6)$.
For $T>0$, we have
$$
\langle k(x) y,J(y) \rangle = \|y\|_X\int_0^\infty k(x)|y(x)| dx,\; y\in D(A).
$$
Then
$$
\begin{array}{lll}
  |\langle k(x) y,J(y) \rangle -\langle k(x) z,J(z) \rangle| &\le & \|y\|_X \int_0^\infty | k(x) (y(x)-z(x))| dx  \\
  &+&  \|y-z\|_X \int_0^\infty |k(x) y(x)| dx
  \\
 &=&\|y\|_X \|B (y-z)\|_X  + \|y-z\|_X \|By\|_X
 \end{array}
$$
which gives $(h_6)$.

$\bullet$ Observation. For $T>0$, we have

$$\begin{array}{lll}
    \langle \mathcal{B} S(t)y,J(S(t)y) \rangle &=& \|S(t)y\|_X^2 + \langle B S(t)y,J(S(t)y) \rangle
  \\
   & \ge& \|y\|_X^2  - \|y\|_X \|B S(t)y\|_X
  \end{array}
$$
then
$$
\begin{array}{lll}
  \int_0^T \langle \mathcal{B} S(t)y,J(S(t)y) \rangle dt & \ge&  \|y\|_X^2  - \|y\|_X \int_0^T \|B S(t)y\|_X dt
  \\
  &\ge& (1-\|k \|_X) \|y\|_X^2.
 \end{array}
$$
Hence $(m_2)$ holds for $\|k \|_X <1.$\\
From Corollary \ref{cor3}, we conclude that the control $v(t)=-\mu {\bf 1}_{\{t\ge 0: \; y(\cdot,t)\ne 0 \}}$ ensures the exponential stability of the system (\ref{E2}).
\end{example}

\begin{example}

 Consider the following system:
\begin{equation}  \label{heat1} \left\{%
\begin{array}{ll} y_{t}(\cdot,t)= y_{xx}(\cdot,t)+v(t) \big (y(\cdot,t)+y_x(\cdot,t) \big ), & \ (x,t)\in [0,1]\times (0,+\infty), \\
        y'(0,t)=y'(1,t)=0, & t\in  (0,+\infty),\\
  y(\cdot,0)=  y_0\in \mathcal{ C}_0([0,1]), & \ x\in [0,1]

\end{array}%
\right.
\end{equation}
where $ u(t)\in \mathbb{R}$ is the control and $y(t)=y(\cdot,t)$ is the state.

\medskip

The state space $X=\mathcal{ C}_0([0,1])$ is equipped with the supremum norm, the  operator   $A=\partial_{xx} $ with domain $\mathcal{ D}(A)=\{y\in \mathcal{C}^2([0,1]): \; y'(0)=y'(1)=0\}$  generates a contraction $C_0-$semigroup
$S(t)$ in $X:=\mathcal{C}_0([0,1])$ (see \cite{eng}, pp. 93-94). The control operator is $ \mathcal{B}=I+ \partial_x, $ then the $A-$boundedness of $\mathcal{B}$  follows from the following  inequalities
$$
|y'(x)|\le \int_0^1 |y''(s)| ds\le \|y\|_{D(A)},\; y\in D(A).
$$
In other words, $\mathcal{B} \in\mathcal{L}( X_1 , X).$

\medskip

We will show that the stabilization assumptions previously considered  are  not all required. In particular, here we only need some elements of the duality set.
For $f\in C_0([0,1])$ we have (\cite{eng}, p. 93): $$\Lambda(f) := \{\varphi= f(s_0) \delta_0: \;  s_0\in [0,1] \; \mbox{is s.t } \;   |f(s_0)|=\|f\|= \displaystyle \max_{s\in [0,1]}|f(s)|\} \subset J(f),$$   where $\delta_{s_0}$ is any point measure  supported by a point $s_0$ where $|f|$ reaches its maximum. For the expression of the full duality map  (see e.g. \cite{bar}, p. 5).

\medskip

Let $y\in D(A) $ and let $y^*=|y(s_0)| \delta_{s_0} \in J(y),$ i.e. $|y(s_0)| =\|y\|.$ Thus  for $s_0\in (0,1)$ we have  $y'(s_0)=0, $ so taking into account the Neumann boundary conditions, we deduce that $y'(s_0)=0$ for every $s_0\in [0,1]$ such that  $|y(s_0)| =\|y\|.$
Thus we have
$$
\langle y',J(y)\rangle=|y'(s_0)| |y(s_0)|=0,\; \forall y\in D(A).
$$
It follows that
$$
\langle \mathcal{B} y,J(y)\rangle =\|y\|^2,\;\forall y\in D(A).
$$
Hence $(h_5)$ and $(h_6)$ are verified for any  element of $\Lambda(f)$.\\
Now, from (\cite{paz}, p. 82) we deduce that $A-\mu\mathcal{B}$ with domain $D(A-\mu\mathcal{B})=D(A)$ is a generator on $X$ for $\mu>0$ small enough. Hence using again that $
\langle \mathcal{B} y,J(y)\rangle =\|y\|^2$ for all $y\in D(A)$, we derive directly  from (\ref{dy}) that the feedback control $v(t)=-\mu {\bf 1}_{\{t\ge 0: \; y(t)+y_x(t)\ne0\}}$ results in an exponentially stable closed-loop system for a small gain control $\mu>0.$
\end{example}

%

\end{document}